\documentclass[twoside,leqno]{article}
\usepackage{amsfonts}
\usepackage{amsmath}
\usepackage{amssymb}
\usepackage{amsthm}
\usepackage{graphicx}
\usepackage{epsfig}
\usepackage{amsmath,amssymb,amsfonts}
\usepackage{hyperref}
\def\bjga#1{\def\thefootnote{}}
\def\mymaketitle#1{\date{\bjga{#1}}
\begin{document}\maketitle\thispagestyle{empty}}
\def\babs{
    \begin{abstractv}}
\def\eabs{\end{abstractv}}

\usepackage{graphicx,caption,subcaption}

%
\theoremstyle{definition}
\newtheorem{definition}{Definition}[section]
\newtheorem{example}[definition]{Example}
\newtheorem{remark}[definition]{Remark}
\newtheorem{problem}[definition]{Problem}
\theoremstyle{theorem}
\newtheorem{theorem}{Theorem}[section]
\newtheorem{proposition}[theorem]{Proposition}
\newtheorem{corollary}[theorem]{Corollary}
\newtheorem{lemma}[theorem]{Lemma}
\def\pas{\par\smallskip}
\def\pasn{\par\smallskip\noindent}
\def\pam{\par\medskip}
\def\pamn{\par\medskip\noindent}
\def\pab{\par\bigskip}
\def\header#1#2#3#4#5{
    \markboth{#3}{#4}
    \setcounter{page}{#1}
    \title{#5}
    \author{#3}
    \date{}
    \mymaketitle{#1-#2}}
\newenvironment{abstractv}{\begin{quote}{\bf Abstract.\ }}{\end{quote}}
\def\msc{{\bf M.S.C. 2010}:\ }
\def\kwd{\\{\bf Key words}:\ }
\def\aua{\par\noindent{\em Author's address:}\pam\noindent}
\def\auas{\par\noindent{\em Authors' addresses:}\pam\noindent}
\def\auac{\par\noindent{\em Authors' address:}\pam\noindent}
\def\bece{\begin{center}}
\def\eece{\end{center}}
\def\bebi{\begin{thebibliography}{40}\setlength{\parskip}{-1mm}}
\def\eebi{\end{thebibliography}}
\def\bibi#1{\bibitem{#1}}
\newcommand{\R}{\mathbb{R}}
\newcommand{\Rt}{\mbox{{\em $\R$}}}
\newcommand{\Rs}{\mbox{\tiny{\R}}}
\newcommand{\C}{\mathbb{C}}
\newcommand{\Cs}{\mbox{\tiny{\C}}}
\newcommand{\Q}{\mathbb{Q}}
\newcommand{\Z}{\mathbb{Z}}
\newcommand{\Zs}{\mbox{\tiny{\Z}}}
\newcommand{\N}{\mathbb{N}}
\def\Ne{\mbox{\em $\N$}}
\newcommand{\Ns}{\mbox{\tiny{\N}}}
\newcommand{\Cbar}{\bar{\mathbf{C}}}
\def\DD{\;\mbox{D}\!\!\!\!\!\!\mbox{I}\;\;\,}
\def\DDs{\mbox{{\scriptsize$\;\mbox{D}\!\!\!\!\!\!\mbox{I}\;\;\,$}}}
\def\noi{\noindent}
\def\qq{\qquad}
\def\mm{\medskip\\}
\def\lg{\langle}
\def\rg{\rangle}
\def\ra{\Righarrow}
\def\lra{\Leftrightarrow}
\def\ri{\rightarrow}
\def\fall{\mbox{ for all }}
\def\di{\displaystyle}
\def\vp{\varphi}
\def\al{\alpha}
\def\ol#1{\overline{#1}}
\def\qed{\hfill$\Box$}
\def\bref#1{(\ref{#1})}
\def\midd{\hspace*{-10pt}\left.\phantom{\di\int}\right|}
\def\text#1{\mbox{#1}}
\def\emph#1{{\em #1}}
\def\textit#1{{\em #1}}
\def\textbf#1{{\bf #1}}
\def\func#1{\mathop{\rm #1}}
\def\limfunc#1{\mathop{\rm #1}}
\def\dint{\displaystyle\int}
\def\dsum{\displaystyle\sum}
\def\dfrac{\displaystyle\frac}
\def\Bbb#1{\mathbb{#1}}
\def\bu{$\bullet$\ }
\def\sta{$\star$\ }
\def\ii{\,\mbox{i}\,}
\def\lb{\linebreak}
\def\pr{{}^{\prime}}
\def\imath{\mbox{i}}
\def\ba{\begin{array}}
\def\ea{\end{array}}
\def\beq{\begin{equation}}
\def\eeq{\end{equation}}
\def\zx#1{\begin{equation}\label{#1}}
\def\zc{\end{equation}}
\def\aru#1{\left\{{\begin{array}{l}#1\end{array}}\right.}
\def\matd#1{\left(\begin{array}{cc}#1\end{array}\right)}
\def\shw{\scriptscriptstyle}
\def\stkdn#1#2{{\mathop{#2}\limits^{}_{#1}}{}}
\def\dn#1#2{\stkdn{{\shw #1}}{#2}{}}
\def\stkup#1#2{{\mathop{#1}\limits^{#2}_{}}{}}
\def\up#1#2{\stkup{#1}{\shw #2}{}}
\def\stkud#1#2#3{{\mathop{#2}\limits^{#3}_{#1}}{}}
\def\ud#1#2#3{\stkud{\shw #1}{#2}{\shw #3}}
\def\stackunder#1#2{\mathrel{\mathop{#2}\limits_{#1}}}
\def\emptyset{\begin{picture}(13,10) \unitlength1pt
    \put(5,3){\circle{7}}\put(.7,-1.8){\line(1,1){10}}\end{picture}}
\def\fiu#1#2#3{\begin{center}                                   
    \includegraphics[scale=#1]{#2.eps}\nopagebreak\\\parbox{10cm}
    {\begin{center}\small\bf #3\end{center}}\end{center} }
\def\fid#1#2#3#4{\begin{center}                                 
    \includegraphics[scale=#1]{#2.eps}\hspace*{.7cm}
    \includegraphics[scale=#1]{#3.eps}
    \nopagebreak\\\parbox{14cm}{\begin{center}
    \small\bf #4\end{center}}\end{center} }
\def\fit#1#2#3#4#5{\begin{center}\begin{tabular}{ccc}
    \includegraphics[scale=#1]{#2.eps} &\hspace*{.4cm}&
    \includegraphics[scale=#1]{#3.eps}\\
    {\small\bf #4}&&{\small\bf #5}\end{tabular}\end{center}}
\pagestyle{myheadings}
\setlength{\textheight}{20cm}
\setlength{\textwidth}{13cm}
\setlength{\oddsidemargin}{18mm}
\setlength{\evensidemargin}{18mm}
\renewcommand{\theequation}{\thesection.\arabic{equation}}
\makeatletter \@addtoreset{equation}{section} \makeatother

\newcommand{\re}{\operatorname{Re}}
\newcommand{\im}{\operatorname{Im}}
\newcommand{\hot}{\operatorname{h.o.t.}}
\newcommand{\holo}{\operatorname{holomorphic}}

\header{1}{20}{P. Connor}       
    {Harmonic parametrizations of surfaces of arbitrary genus} 
    {Harmonic parametrizations of surfaces of \\arbitrary genus}  
%
%
\babs
    The Weierstrass representation for minimal surfaces in $\R^3$ provides a flexible method for constructing minimal surfaces of arbitrary genus.  The topological limitations of minimal surfaces interfere with this providing a more general geometric modeling tool.  Minimal surfaces lie in the larger class of harmonic surfaces, which in general don't have the same topological limitations of minimal surfaces and can have complicated embedded ends.  In this paper we demonstrate the flexibility of using the Weierstrass representation for harmonic surfaces to combine embedded harmonic ends together to construct embedded harmonic surfaces of arbitrary genus.
\eabs
\msc
    53A10, 49Q05, 53C42.       
\kwd
    harmonic surface; minimal surface.  
%
\section{Introduction}
    In \cite{we9}, Weber outlined a method for constructing surfaces in $\R^3$ whose parametrization is given by harmonic coordinate functions.  The most well known surfaces of this type are minimal surfaces.  As shown in \cite{clw1}, there is a vast array of embedded harmonic ends, which can be combined together in surprising ways to construct embedded surfaces.  

We call a surface {\it harmonic} if it is possible to find a parametrization whose coordinate functions are harmonic.  Many harmonic surfaces are graphs over the xy-plane, for example the surface in figure \ref{figure:g0(2,2,4)} with parametrization $f(x,y)=(x,y,x^3/3-xy^2)$.
\begin{figure}[h]
	\centerline{ 
		\includegraphics[height=2.5in]{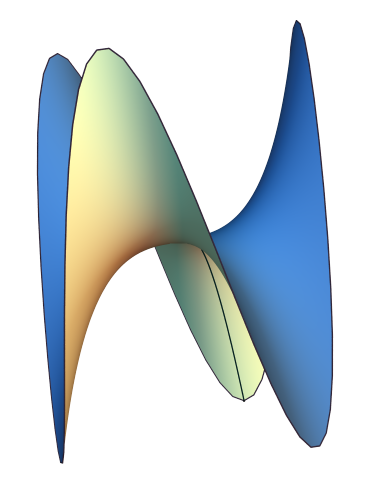}
			}
	\caption{Harmonic graph over the xy-plane}
	\label{figure:g0(2,2,4)}
\end{figure}
  Among the most well known harmonic surfaces are minimal surfaces, with the Weierstrass representation for a minimal surface providing a harmonic parametrization.  This method concocts the coordinate functions of the parametrization by taking the real part of the integral of meromorphic one-forms.  
  
In the case of minimal surfaces, there is a well-defined limiting normal vector at each end.  In \cite{ka6}, Karcher illustrated a method of constructing minimal surfaces with prescribed geometry, utilizing the Weierstrass representation and the nice behavior of the normal vector.  Harmonic surfaces, in general, don't have a well-defined limiting normal vector at each end.  However, the Weierstrass representation still offers a convenient way to construct harmonic surfaces, and the usual obstacle of solving the period problem isn't an issue with harmonic surfaces.
  

There are interesting applications of harmonic surfaces to the field of computer aided geometric design.  In \cite{mu1}, Monterde and Ugail used harmonic B\'e{}zier surfaces to generate surfaces with prescribed boundaries.  Harmonic surfaces provide a rich source of easy to construct non-positively curved surfaces that can be used to solve free boundary value problems. 

There are many different types of embedded harmonic ends which can be combined in different combinations to construct embedded surfaces with desired end behavior.  These ends can be expressed using the Weierstrass representation.  In turn, the Weierstrass representations of a group of ends can be easily combined to construct the Weierstrass representation of a surface with that group of ends.

Additionally, it is possible to construct harmonic surfaces of arbitrary genus.  Other than keeping the surface embedded, there is no restriction on the location and size of the handles.  For example, there is only a genus zero minimal catenoid.  Whereas one can construct harmonic catenoids of arbitrary genus, with no restrictions on the location of the handles.  This provides a great deal of flexibility to design non-positively curved surfaces of any genus.

There are numerous modeling applications of minimal surfaces in fields such as architecture, aviation, biology, and nano structures \cite{gdm,se1,su1,wa1}, and any minimal surface can be approximated by a harmonic surface.  Scherk's doubly periodic minimal surface, for example, can be modified by adding a variety of handles with different locations and sizes.  The resulting harmonic surface won't necessarily remain a minimal surface, but it can be viewed as an approximation of a minimal surface, with the ends asymptotic to minimal ends.  Harmonic surfaces provide flexible approximations of minimal surfaces that can be used in any field that utilizes minimal surfaces.

This paper demonstrates how to construct examples of harmonic surfaces with prescribed geometry.  In section \ref{sec2}, we discuss some background information and terms that will be used in the paper.  In section \ref{sec3}, we construct an embedded genus zero surface with a finite number of parallel planes by gluing planes between the top and bottom ends of a catenoid.  In section \ref{sec4}, we illustrate how to construct examples of arbitrary genus with two orthogonal symmetry planes.  In section \ref{sec5}, we show the flexibility of harmonic surfaces by adding handles in different locations to Scherk's doubly periodic surface.  In section \ref{sec6}, we illustrate how to construct harmonic tori with no restrictions on the placement of the handle.
%
%
\section{Background information}
\label{sec2}
The study of harmonic immersions in $\R^3$ began in the 1960's with a series of papers by Klotz \cite{kl2,kl3,kl4, kl5,kl1} that studied the Gauss map of harmonic immersions and its quasiconformal properties.  Alarc\'o{}n and L\'o{}pez \cite{al1}  continued the development of the global theory of harmonic immersions, proving that a complete harmonic immersion in $\R^3$ has finite  $L^2$ norm of the shape operator ($\int |S|^2\, dA <\infty$) if and only if the domain of the surface is conformally a compact Riemann surface with finitely many punctures and the Gauss map extends continuously to the punctures.  

Let $X$ be a compact Riemann surface, $X'=X-\{p_1,\ldots,p_m\}$, and $\omega_1,\omega_2, \omega_3$ meromorphic one-forms on $X$ that are holomorphic on $X'$.  If 
\[
\re\int_{\gamma}\left(\omega_1,\omega_2,\omega_3\right)=(0,0,0)
\]
where $\gamma$ is any closed curve in $X'$ (period problem), then 
\[
f(z)=\re\int^z(\omega_1,\omega_2,\omega_3)
\]
defines a {\it harmonic map} from $X'$ into $\R^3$ with ends at $p_1,\ldots,p_m$.  It's image $M=f(X')$ is called a {\it harmonic surface}, and the pair $\left(X',(\omega_1,\omega_2,\omega_2)\right)$ is referred to as the {\it Weierstrass representation} of $M$.  If one also requires that 
\[
\omega_1^2+\omega_2^2+\omega_3^2=0
\]
then $f$ is a conformal map, and the image $f(X')$ is a minimal surface.

Let $n^j=\text{max}\{n_1^j,n_2^j,n_3^j\}$, where $n_i^j$ is order of the pole of $\omega_i$ at $p_j$.  Then $n^j$ is the {\it order} of the end at $p_j$.  Note that ends with poles of different orders - for example $(0,1,2)$ and $(2,0,1)$ - may differ by a real affine transformation.  In this case, we say the ends are {\it affinely equivalent}.  Note that $n^j$ is the same for all affinely equivalent ends.

An end is in {\it reduced form} if the pole orders $n_k^j$ of $\omega_k$ at $p_j$ satisfy $n_1^j\leq n_2^j\leq n_3^j$ and $\left(n_1^j,n_2^j,n_3^j\right)$ is minimal in lexicographic ordering among all affinely equivalent ends.  The {\it type} of an end at $p_j$ is the tuple $\left(n_1^j,n_2^j,n_3^j\right)$ of an affinely equivalent end in reduced form.  For example, the Weierstrass data
\[
\left(\omega_1,\omega_2,\omega_3\right)=\left(\frac{1}{z^2},i,\frac{1}{z}\right)dz
\]
and
\[
\left(\omega_1,\omega_2,\omega_3\right)=\left(\frac{1}{z^2},\frac{1}{z^2}+i,\frac{1}{z}\right)dz
\]
both have ends at $z=0$ and $z=\infty$, of type $(0,1,2)$.  See figures \ref{figure:(0,1,2)a} and \ref{figure:(0,1,2)b}.   The Weierstrass data 
\[
\left(\omega_1,\omega_2,\omega_3\right)=\left(\frac{1}{z^2},\frac{i}{z^2}+i,\frac{1}{z}\right)dz
\]

has an end at $z=0$ of type $(1,2,2)$ and an end at $z=\infty$ of type $(0,1,2)$.  See figure \ref{figure:(0,1,2)c}.
\begin{figure}[t]
    \centering
    \begin{subfigure}[b]{0.29\textwidth}
       \centering
        \includegraphics[width=.7\textwidth]{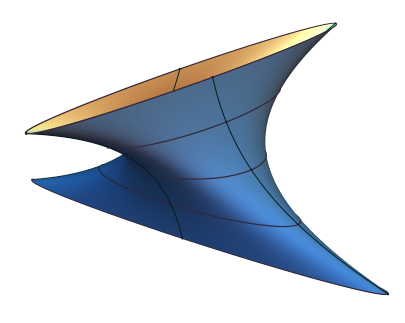}
        \caption{Ends of type $(0,1,2)$}
        \label{figure:(0,1,2)a}
    \end{subfigure}
    ~ 
    \begin{subfigure}[b]{0.31\textwidth}
    	    \centering
        \includegraphics[width=1.1\textwidth]{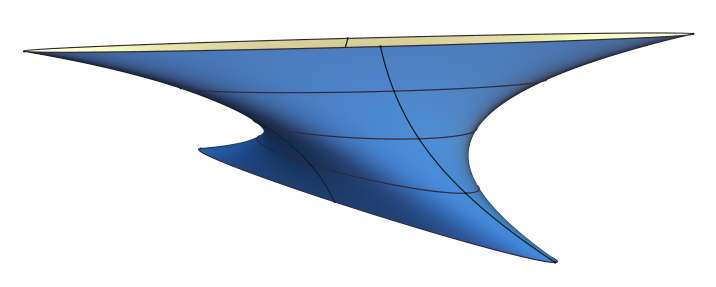}
        \caption{Ends of type $(0,1,2)$}
        \label{figure:(0,1,2)b}
    \end{subfigure}
        ~ 
    \begin{subfigure}[b]{0.35\textwidth}
    	    \centering
        \includegraphics[width=.9\textwidth]{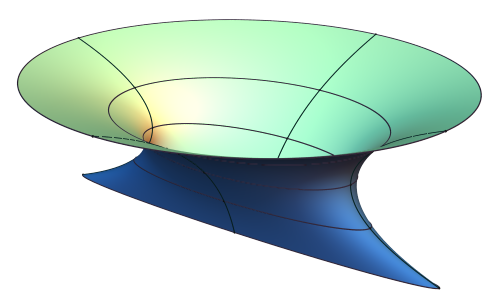}
        \caption{Ends of type $(0,1,2), (1,2,2)$}
        \label{figure:(0,1,2)c}
    \end{subfigure}
    \caption{}\label{figure:(0,1,2)}
\end{figure}

It was proven in \cite{clw2} that if $f$ is an immersion then the total Gauss curvature of $M$ is 
\[
\int_{X'} KdA=2\pi\chi(X)-2\pi\sum_{j=1}^m(n^j-1),
\]
where $\chi(X)$ is the Euler characteristic of $X$.  

The Gauss map $N:X\rightarrow S^2$ is given by 
\[
N(z)=\frac{\im\left(\omega_2(z)\overline{\omega_3(z)},\omega_3(z)\overline{\omega_1(z)},\omega_1(z)\overline{\omega_2(z)}\right)}{\left|\left|\im\left(\omega_2(z)\overline{\omega_3(z)},\omega_3(z)\overline{\omega_1(z)},\omega_1(z)\overline{\omega_2(z)}\right)\right|\right|},
\]
and it extends continuously to $p_j$ iff the end at $p_j$ is of type $(n_1^j,n_2^j,n_2^j)$.  Note that an end at $p_j$ is of type $(n_1^j,n_2^j,n_2^j)$ iff there is a real affine transformation $R$ such that the pole orders of the end at $p_j$ of $\tilde{f}=R\circ f$ satisfy $n_1^j<n_2^j=n_3^j$, with 
\[
(\tilde{\omega_1},\tilde{\omega_2},\tilde{\omega_3})=\left(\frac{a_1+ib_1+O(z)}{(z-p_j)^{n_1^j}},\frac{1+O(z)}{(z-p_j)^{n_2^j}},\frac{i+O(z)}{(z-p_j)^{n_2^j}}\right)dz
\]
in a neighborhood of $p_j$.

The harmonic graph in figure \ref{figure:g0(2,2,4)} has Weierstrass representation

\[
\left(\omega_1,\omega_2,\omega_3\right)=\left(1,-i,z^2\right)dz.
\]
It has one end at $z=\infty$ of type $(2,2,4)$, and so it doesn't have a well-defined limiting normal at its end.  The surfaces in figures \ref{figure:(0,1,2)a} and \ref{figure:(0,1,2)b} also don't have a well-defined limiting normal at their ends.  The surface in figure \ref{figure:(0,1,2)c} has a well-defined limiting normal at $z=0$ but not at $z=\infty$.

Lemma 2.11 in \cite{clw1} states that an end of type $(n_1,n_2,n_2)$ with $n_2>2$ cannot be embedded.  Hence, the only embedded harmonic ends for which the Gauss map extends continuously to the end are planes and catenoids - affinely equivalent to an end of type $(0,2,2)$ or $(1,2,2)$, respectively.

\subsection{Solving the period problem}
\label{ssecperiod}
Solving the period problem for a minimal surface is often a serious technical challenge, in particular for higher genus surfaces.  The technical difficulties arising from the period problem for harmonic surfaces are dealt with by adding combinations of holomorphic one-forms to $\omega_1, \omega_2$, and $\omega_3$ as necessary to close the periods.  This is done by utilizing the fact that there is a one-to-one correspondence between the real parts of the periods of a compact Riemann surface $X$ and a basis for holomorphic one-forms on $X$.  In section four, this is implemented on hyperelliptic Riemann surfaces of arbitrary genus.  In section five, this is implemented on a rectangular torus using the holomorphic one-form $dz$.

If $X$ is a hyperelliptic Riemann surface given by the equation
\[
w^2=\prod_{k=1}^{2g+2}(z-q_k)
\]
then Corollary 1 on page 98 of \cite{fkr1} states that ``the $g$ differentials 
\[
\frac{z^jdz}{w},\, j=0,1,\ldots,g-1,
\]
form a basis for the abelian differentials of the first kind on $X$."  Thus, if $\omega_1, \omega_2$, and $\omega_3$ are one-forms on $X$ and $A_j,B_j,$ $j=1,2,\ldots, g$ is a canonical homology basis for $X$ then there exist $\lambda_k\in\C$, $k=1,2,\ldots, 3g$ such that
\[
\int_{A_j}\tilde{\omega_k}=\int_{B_j}\tilde{\omega_k}
=0
\]
for $j=1,2,\ldots,g$ and $k=1,2,3$, with
\[
\tilde{\omega_k}=\omega_k+\sum_{i=1}^g\frac{\lambda_{(k-1)g+i}z^{i-1}dz}{w}.
\]
The modified one-forms $\tilde{\omega_k}$ have no periods on $X$, and they have the same end behavior at the poles of the $\omega_k$ because the terms added to each $\omega_k$ are holomorphic on $X$.

\section{Stacking parallel planes}
\label{sec3}
One of the simplest types of embedded ends is a planar end.  Riemann's minimal surface is a singly periodic embedded minimal surface consisting of a sequence of horizontal planar ends with consecutive planes connected by catenoid-shaped necks.  It was proven in \cite{we5} that a finite Riemann minimal surface with a horizontal planar end and two catenoid ends can't be embedded because the catenoid ends must have non-vertical limiting normals.  It is possible to create an embedded harmonic example with any number of horizontal planar ends stacked between two catenoid ends with vertical limiting normals.  See figure \ref{figure:5planes} for an example with 5 planar and 2 catenoid ends.
\begin{figure}[h]
	\centerline{ 
		\includegraphics[height=2.3in]{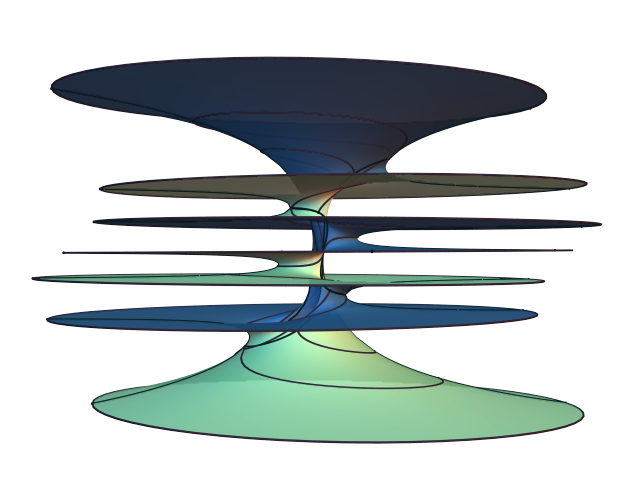}
		\includegraphics[height=2.3in]{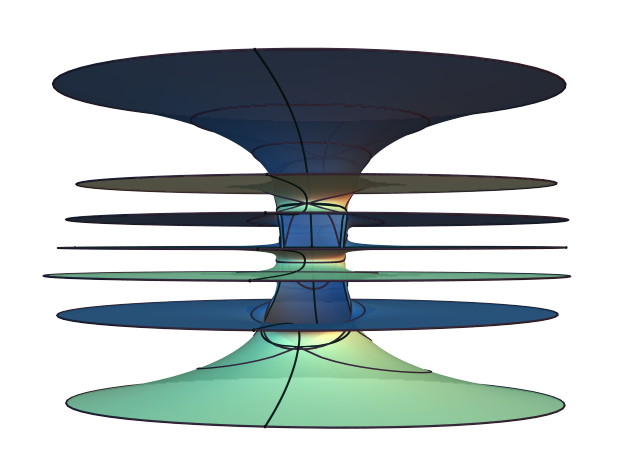}
			}
	\caption{Two views of an embedded harmonic surface with 5 planar and 2 catenoidal ends}
	\label{figure:5planes}
\end{figure}
\begin{theorem}
For each positive integer $n$ there exists a genus zero embedded harmonic surface with $2n-1$ ends of type $(0,2,2)$ (planar ends) stacked between two ends of type $(1,2,2)$ (catenoids), each end with vertical limiting normal.
\end{theorem}

\begin{proof}
An end of type $(1,2,2)$ at $z=a$, after rotating and rescaling if necessary, is of the form
\[
(\omega_1,\omega_2,\omega_3)=\left(\frac{\alpha}{(z-a)^2},\frac{i}{(z-a)^2},\frac{1}{z-a}\right)dz
\]
with $\alpha\in\R$ and limiting normal $N(a)=(0,0,\alpha/|\alpha|)$.  An end of type $(0,2,2)$ at $z=a$, after rotating and rescaling if necessary, is of the form
\[
(\omega_1,\omega_2,\omega_3)=\left(\frac{\beta}{(z-a)^2},\frac{i}{(z-a)^2},h(z)\right)dz
\]
with $\beta\in\R$, $h$ holomorphic at $z=a$, and limiting normal $N(a)=(0,0,\beta/|\beta|)$.  

Let's construct the surface so that the ends are placed at $k\in\Z$, $-n\leq k\leq n$.  Place a catenoid end at $z=-n$ opening downward, a catenoid end at $z=n$ opening upward, and planar ends at $z=k$, $-n<k<n$.  We want the ends to be in ascending order from $z=-n$ to $z=n$.  This is achieved by setting 
\[
\omega_3=\left(\frac{1}{z+n}-\frac{1}{z-n}\right)dz
\]
so that the third coordinate of $f$ is given by 
\[
f_3(z)=\frac{1}{2}\log\left(\frac{|z+n|^2}{|z-n|^2}\right).
\]
Fix the limiting normal $N(0)=(0,0,1)$.  The limiting normal at successive levels must alternate between $(0,0,1)$ and $(0,0,-1)$.  Hence, the planar ends at $z=\pm k$ are given by the data
\[
(\omega_1,\omega_2,\omega_3)=\left(\frac{(-1)^k}{(z\pm k)^2},\frac{i}{(z\pm k)^2},\frac{1}{z+n}-\frac{1}{z-n}\right)dz.
\]
Combining the data from each end, we get $X'=\overline{\C}-\{0,\pm 1,\pm 2,\ldots,\pm n\}$ and
\[
\begin{split}
\omega_1=&\left(\frac{1}{z^2}+\sum_{k=1}^n (-1)^k\left(\frac{1}{(z-k)^2}+\frac{1}{(z+k)^2}\right)\right)dz\\
\omega_2=&\left(\frac{i}{z^2}+i\sum_{k=1}^n \left(\frac{1}{(z-k)^2}+\frac{1}{(z+k)^2}\right)\right)dz\\
\omega_3=&\left(\frac{1}{z+n}-\frac{1}{z-n}\right)dz.
\end{split}
\]
The periods can be computed as residues since the domain $X'$ is a punctured sphere, and the residues of $\omega_1, \omega_2$, and $\omega_3$ are all real.  Thus, the period problem is automatically solved, and $(\omega_1,\omega_2,\omega_3)$ is the Weierstrass data for a harmonic surface with catenoid ends at $z=\pm n$ and planar ends at $z=k$, $n< k< n$.  The end at $z=k$ is embedded with limiting normal $N(k)=(0,0,(-1)^k)$, for $-n\leq k\leq n$.  
\end{proof}

\section{Constructing higher genus examples}
\label{sec4}
There is a simple procedure for producing higher genus surfaces with two orthogonal symmetry planes.  Surfaces with two ends of this type can be constructed by combining ends of type $(0,0,1)$, $(0,1,2)$, $(0,2,3)$, $(1,2,2)$, and $(2,2,3)$.  In this section, we construct a template for surfaces of arbitrary genus with two ends of type $(0,1,2)$.  The prototype for this is the genus zero surface defined on $\C-\{0\}$ by 
\[
\left(\omega_1,\omega_2,\omega_3\right)=\left(\frac{1}{z^2},i,\frac{1}{z}\right)dz
\]
with ends of type $(0,1,2)$ at $z=0$ and $z=\infty$.  See figure \ref{figure:(0,1,2)a}.

Let 
\[
X=\left\{(z,w)\in\C^2|w^2=z\prod_{k=1}^{2n}(z-a_k)\right\},
\]
a genus n hyperelliptic Riemann surface.  Assume that $a_k\in\R$ for $k=1,2,\ldots,2n$.  Then, $\tau(z,w)=(\overline{z},-\overline{w})$ and $\sigma(z,w)=(\overline{z},\overline{w})$ are automorphisms of $X$.  The pairs of points $(a_{2k-1},0), (a_{2k},0)\in X$ will correspond to a handle on the surface we are constructing.  Assuming the period problem is solved, the one forms 
\[
\begin{split}
\omega_1&=\frac{\prod_{k=1}^n(z-a_{2k-1})}{w}dz=\frac{1}{\sqrt{z}}\prod_{k=1}^n\frac{\sqrt{z-a_{2k-1}}}{\sqrt{z-a_{2k}}}dz\\
\omega_2&=\frac{i\prod_{k=1}^n(z-a_{2k})}{zw}dz=\frac{i}{z^{3/2}}\prod_{k=1}^n\frac{\sqrt{z-a_{2k}}}{\sqrt{z-a_{2k-1}}}dz\\
\omega_3&=\frac{1}{z}dz\\
\end{split}
\]
provide the Weierstrass data for a harmonic surface with domain $X$ and ends of type $(0,1,2)$ at $(0,0)$ and $(\infty,\infty)$, whose map is
\[
f(z)=\re\int^z\left(\omega_1,\omega_2,\omega_3\right).
\]
As discussed in section \ref{ssecperiod}, the period problem is solved by adding multiples of the holomorphic one-forms 
\[
\frac{z^{k-1}dz}{w},\;\;k=1,2,\ldots,n
\]
to $\omega_1$ and $\omega_2$ to solve the period problem.  As they are holomorphic on $X$, they won't affect the end behavior of the surface we are constructing.  There are no periods of $\omega_3$.  This yields the following theorem.

\begin{theorem}
For each $(a_1,a_2,\cdots,a_{2n})\in\R^{2n}$ with $a_j\neq a_k$ if $j\neq k$ there exist $\lambda_k\in\C$, $k=1,2,\ldots,2n$ such that
\[
X=\left\{(z,w)\in\C^2|w^2=z\prod_{k=1}^{2n}(z-a_k)\right\}
\]
and
\[
\begin{split}
\omega_1&=\left(\frac{1}{\sqrt{z}}\prod_{k=1}^n\frac{\sqrt{z-a_{2k-1}}}{\sqrt{z-a_{2k}}}+\sum_{k=1}^n\frac{\lambda_k z^{k-1}}{w}\right)dz\\
\omega_2&=i\left(\frac{1}{z^{3/2}}\prod_{k=1}^n\frac{\sqrt{z-a_{2k}}}{\sqrt{z-a_{2k-1}}}+\sum_{k=1}^n\frac{\lambda_{n+k} z^{k-1}}{w}\right)dz\\
\omega_3&=\frac{1}{z}dz
\end{split}
\]
is the Weierstrass data for an embedded harmonic surface of genus $n$ with two ends of type $(0,1,2)$.
\label{thm2}
\end{theorem}

Note that $\tau^*(\omega_1,\omega_2,\omega_3)=(-\overline{\omega_1},\overline{\omega_2},\overline{\omega_3})$ and $\sigma^*(\omega_1,\omega_2,\omega_3)=(\overline{\omega_1},-\overline{\omega_2},\overline{\omega_3})$.  Thus, $\tau$ and $\sigma$ induce reflections in planes parallel to the $x_2x_3$-plane and $x_1x_3$-plane, respectively.  Hence, the image $f(X)$ has two symmetry planes and can be split into four simply connected pieces.  One fundamental domain is $X_1=\{(z,w)\in X|z,w>0\}$ - see figure \ref{figure:(0,1,2)funddomb}.  The boundary curves on the left and right are the images of the positive and negative real axis, respectively.  So, if $a_k>0$ for all $k$ or $a_k<0$ for all $k$ then the handles will all lie along one of the two symmetry curves.  See figure \ref{figure:g4(0,1,2)funddom} for an example with the handles added along the image of the positive real axis.

\begin{figure}[h]
    \centering
    \begin{subfigure}[b]{0.45\textwidth}
       \centering
        \includegraphics[width=.6\textwidth]{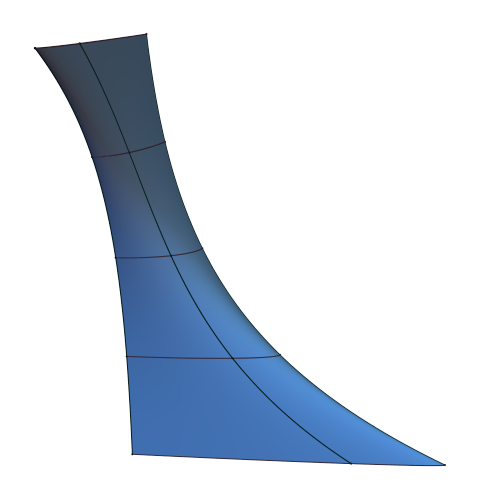}
        \caption{$f(X_1)$ before adding handles}
        \label{figure:(0,1,2)funddomb}
    \end{subfigure}
    ~ 
    \begin{subfigure}[b]{0.45\textwidth}
    	    \centering
        \includegraphics[width=\textwidth]{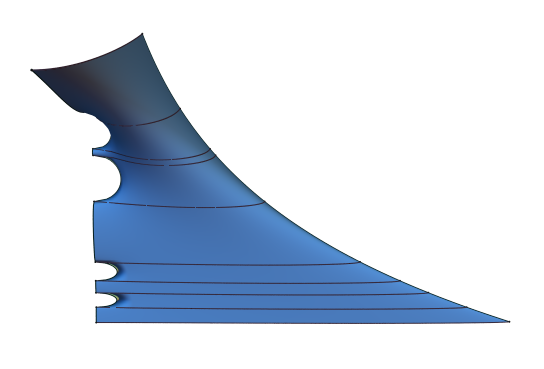}
        \caption{$f(X_1)$ with $a_k>0$}
        \label{figure:g4(0,1,2)funddom}
    \end{subfigure}
    ~ 
    \caption{}\label{figure:(0,1,2)funddom}
\end{figure}
\begin{example}
Let's examine the surface when $n=4$ and 
\[
(a_1,a_2,a_3,a_4,a_5,a_6,a_7,a_8)=\left(\frac{1}{5},\frac{1}{4},\frac{3}{10},\frac{2}{5},1,2,2.2,4\right).
\]
See figure \ref{figure:(0,1,2)genus4widemouthc}.  The handles will be added along the left boundary curve of $f(X_1)$, as in figure \ref{figure:g4(0,1,2)funddom}.  From the bottom to the top, the handles will lie between the images of the pairs $(a_1,a_2)$, $(a_3,a_4)$, $(a_5,a_6)$, and $(a_7,a_8)$.  In this case,
\[
\begin{split}
\omega_1&=\left(\frac{\sqrt{z-1/5}\sqrt{z-3/10}\sqrt{z-1}\sqrt{z-2.2}}{\sqrt{z}\sqrt{z-1/4}\sqrt{z-2/5}\sqrt{z-2}\sqrt{z-4}}+\frac{\lambda_1+\lambda_2 z+\lambda_3 z^2+\lambda_4 z^3}{w}\right)dz\\
\omega_2&=i\left(\frac{\sqrt{z-1/4}\sqrt{z-2/5}\sqrt{z-2}\sqrt{z-4}}{z^{3/2}\sqrt{z-1/5}\sqrt{z-3/10}\sqrt{z-1}\sqrt{z-2.2}}+\frac{\lambda_5+\lambda_6 z+\lambda_7 z^2+\lambda_8 z^3}{w}\right)dz\\
\omega_3&=\frac{1}{z}dz.
\end{split}
\]
\begin{figure}[h]
	\centerline{ 
		\includegraphics[height=2in]{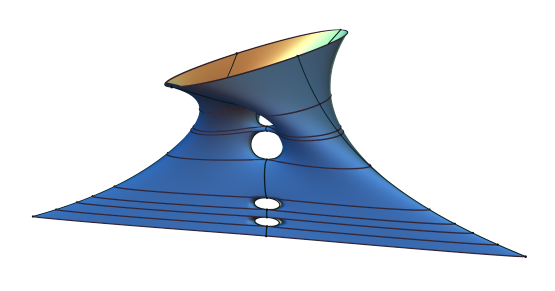}
			}
	\caption{$(a_1,a_2,a_3,a_4,a_5,a_6,a_7,a_8)=\left(\frac{1}{5},\frac{1}{4},\frac{3}{10},\frac{2}{5},1,2,2.2,4\right)$}
	\label{figure:(0,1,2)genus4widemouthc}
\end{figure}
\end{example}

The symmetries of the surface reduce the period problem to the equations
\[
\re\int_0^{a_1}\omega_j=\re\int_{a_{k}}^{a_{k+1}}\omega_j=0
\]
for $j=1,2$ and $k=1,2,3,4,5,6,7$.  This is a system of linear equations in the $\lambda_k$ variables which, by theorem \ref{thm2}, is solvable.  Solving numerically, one finds the system is solved when 
\[
\begin{split}
\lambda_1&=-0.04470938,\, \lambda_2=0.32501167, \lambda_3=-0.64421279,\, \lambda_4=0.25789177,\\
\lambda_5&=-0.48150461,\, \lambda_6=2.46598103,\,\lambda_7=-2.99070994,\, \lambda_8=0.81972586.
\end{split}
\]

\begin{example}
\label{example2}
Let's examine the surface when $n=4$ and 
\[
(a_1,a_2,a_3,a_4,a_5,a_6,a_7,a_8)=\left(\frac{1}{5},\frac{1}{4},\frac{3}{10},\frac{2}{5},-\frac{5}{2},-\frac{10}{3},-4,-5\right).
\]  
See figure \ref{figure:genus4widemoutha}.  With this example, the two lower handles, between the images of the pairs $(a_1,a_2)$ and $(a_3,a_4)$, will be added along the left boundary curve of $f(X_1)$ in figure \ref{figure:(0,1,2)funddomb}.  The two upper handles, between the images of the pairs $(a_5,a_6)$ and $(a_7,a_8)$, will be added along the right boundary curve of $f(X_1)$ in figure \ref{figure:(0,1,2)funddomb}.  In this case,
\[
\begin{split}
\omega_1&=\left(\frac{\sqrt{z-\frac{1}{5}}\sqrt{z-\frac{3}{10}}\sqrt{z+\frac{5}{2}}\sqrt{z+4}}{\sqrt{z}\sqrt{z-\frac{1}{4}}\sqrt{z-\frac{2}{5}}\sqrt{z+\frac{10}{3}}\sqrt{z+5}}+\frac{\lambda_1+\lambda_2 z+\lambda_3 z^2+\lambda_4 z^3}{w}\right)dz\\
\omega_2&=i\left(\frac{\sqrt{z}\sqrt{z-\frac{1}{4}}\sqrt{z-\frac{2}{5}}\sqrt{z+\frac{10}{3}}\sqrt{z+5}}{z^{3/2}\sqrt{z-\frac{1}{5}}\sqrt{z-\frac{3}{10}}\sqrt{z+\frac{5}{2}}\sqrt{z+4}}+\frac{\lambda_5+\lambda_6 z+\lambda_7 z^2+\lambda_8 z^3}{w}\right)dz\\
\omega_3&=\frac{1}{z}dz.
\end{split}
\]
\begin{figure}[h]
	\centerline{ 
		\includegraphics[height=1.97in]{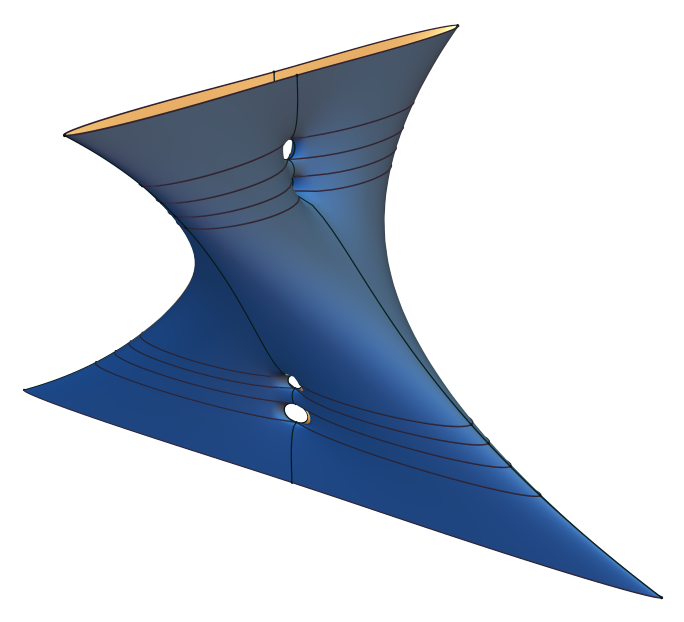}
			}
	\caption{$(a_1,a_2,a_3,a_4,a_5,a_6,a_7,a_8)=\left(\frac{1}{5},\frac{1}{4},\frac{3}{10},\frac{2}{5},-\frac{5}{2},-\frac{10}{3},-4,-5\right)$}
	\label{figure:genus4widemoutha}
\end{figure}

The symmetries of the surface reduce the period problem to the equations
\[
\re\int_{\alpha}\omega_j=\re\int_0^{a_1}\omega_j=\re\int_{a_{k}}^{a_{k+1}}\omega_j=0
\]
for $j=1,2$, $k=1,2,3,5,6,7$, where $\alpha$ is the half-circle from $a_4$ to $a_5$ in the upper half plane.  In example $4.1$, the corresponding period between $a_4$ and $a_5$ could be calculated as an integral over $(a_4,a_5)\in\R$.  Here, $a_4=2/5$ and $a_5=-5/2$, and so if the integral between $a_4$ and $a_5$ is restricted to the real axis then it would need to go through a singularity at $z=0$ or $z=\infty$.  

The period problem is a system of equations in the $\lambda_k$ variables which, by theorem \ref{thm2}, is solvable.  Solving numerically, one finds that system is solved when 
\[
\begin{split}
\lambda_1&=-0.51419509,\, \lambda_2=-0.0470358,\,\lambda_3=0.85173589,\, \lambda_4=0.25097214,\\
\lambda_5&=-0.4182869,\, \lambda_6=1.41955981,\,\lambda_7=0.07839301,\, \lambda_8= -0.08569918.
\end{split}
\]
\end{example}

\begin{example}
If you prefer to have an end of type $(0,2,3)$ at $(0,0)$ instead of type $(0,1,2)$ then a small modification to $\omega_3$ is all that is needed.  Keep $\omega_1$ and $\omega_2$ as in example \ref{example2}, and let
\[
\omega_3=\left(\frac{1}{z}+\frac{1}{z^2}\right)dz.
\]
This produces an embedded, genus four surface with ends of type $(0,1,2)$ and $(0,2,3)$.  See figure \ref{figure:genus4(0,2,1)(2,0,3)}.

\begin{figure}[h]
	\centerline{ 
		\includegraphics[height=3in]{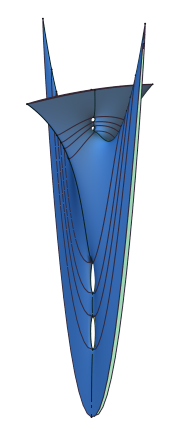}
			}
	\caption{Genus four surface with ends of type $(0,1,2)$ and $(0,2,3)$}
	\label{figure:genus4(0,2,1)(2,0,3)}
\end{figure}
\end{example}

\section{Adding handles to Scherk's doubly periodic minimal surface}
\label{sec5}
As proven by Weber and Wolf in \cite{ww4}, for each $n\geq 0$ there exists an embedded, doubly periodic, genus $n$ minimal surface with four Scherk ends in the quotient.  Each surface lies in a one-parameter family, with the parameter given by the angle between the top and bottom ends.  Thus, there is very little flexibility as to how one adds the handles to Scherk's doubly periodic surface.  If we drop the minimal condition then it is possible to produce embedded, doubly periodic harmonic surfaces with Scherk ends of arbitrary genus, where you can choose the precise location and direction of the handles.  

Let's consider how we can add two handles to Scherk's surface.  Assume that the top and bottom ends are orthogonal and that there are two orthogonal symmetry planes.  This allows for working with a genus zero fundamental domain that has two ends.  We can explicitly write down the Weierstrass representation for these surfaces, proceeding as if we are creating a minimal surface with 
\[
\begin{split}
\omega_1&=\frac{1}{2}\left(\frac{1}{g(z)}-g(z)\right)\omega_3\\
\omega_2&=\frac{i}{2}\left(\frac{1}{g(z)}+g(z)\right)\omega_3\\
\end{split}
\]
where $g$ is the composition of stereographic projection with the Gauss map.  The normal vector will be vertical at the top and bottom of the two handles on the surface.  

One representation of Scherk's surface is given by
\[
\begin{split}
g(z)&=\frac{\sqrt{z-a_1}}{\sqrt{z+a_1}}\\
\omega_3&=\frac{1}{z}dz,
\end{split}
\]
which places Scherk ends at $0$ and $\infty$.  The surface normal points up at $z=-a_1$ and down at $z=a_1$.  See figure \ref{figure:sscherk}.

\begin{figure}[h]
    \centering
    \begin{subfigure}[b]{0.45\textwidth}
        \includegraphics[width=\textwidth]{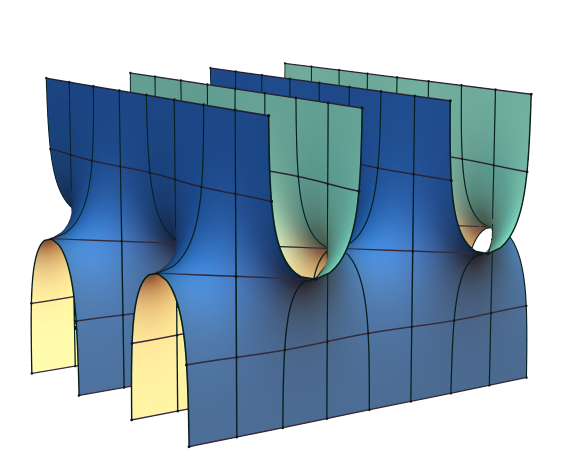}
        \caption{Scherk's doubly periodic surface}
        \label{figure:sscherk}
    \end{subfigure}
    ~ 
    \begin{subfigure}[b]{0.45\textwidth}
    	    \centering
        \includegraphics[width=.3\textwidth]{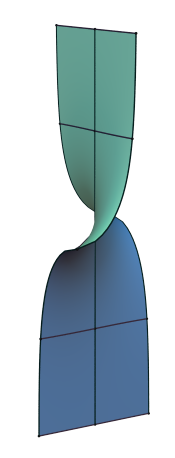}
        \caption{Image of the fundamental domain}
        \label{figure:Scherkfund}
    \end{subfigure}
    ~ 
    \caption{}\label{figure:scherk}
\end{figure}

The easiest places to add handles are along the two vertical curves in the symmetry planes, as shown in figure \ref{figure:Scherkfund}.  These curves are the image of the positive and negative real axes.  Let's demonstrate adding handles along the positive real axis.  Note that it isn't difficult to add handles along both curves.

Consider the Weierstrass data
\[
\begin{split}
g(z)&=\frac{\sqrt{z-a_1}\sqrt{z-a_3}\sqrt{z-a_5}}{\sqrt{z-a_2}\sqrt{z-a_4}\sqrt{z-b_1}}\\
\omega_3&=\frac{1}{z}dz
\end{split}
\]
with $b_1<0<a_1<a_2<a_3<a_4<a_5$.  One handle will be between $a_1$ and $a_2$.  The other handle will be between $a_3$ and $a_4$.  Note the surface normal points down at $a_1$, $a_3$, and $a_5$, and the surface normal points up at $a_2$, $a_4$, and $b_1$.  

In order to have orthogonal ends, we need $g(0)=i$, which forces 
\[
b_1=-\frac{a_1a_3a_5}{a_2a_4}.
\]
The period problem reduces to solving the equations
\[
\re\int_{a_1}^{a_2}\omega_j=\re\int_{a_2}^{a_3}\omega_j=\re\int_{a_3}^{a_4}\omega_j=\re\int_{a_4}^{a_5}\omega_j=0
\]
for $j=1,2$.  If we were attempting to find a minimal surface then we would solve for $a_1,a_2,a_3, a_4$, and $a_5$.  There is one solution.  Instead, using the technique from section \ref{ssecperiod} for solving the period problem on a hyperelliptic Riemann surface, we can add multiples of the holomorphic forms 
\[
\frac{dz}{w}, \frac{zdz}{w}
\]
where
\[
w=\sqrt{z-a_1}\sqrt{z-a_2}\sqrt{z-a_3}\sqrt{z-a_4}\sqrt{z-a_5}\sqrt{z-b_1}
\]
to $\omega_1$ and $\omega_2$ to solve the period problem, yielding the following result.

\begin{theorem}
Let $(a_1,a_2,a_3,a_4,a_5)\in\R^5_+$ with $0<a_1<a_2<a_3<a_4<a_5$,
\[
\begin{split}
b_1&=-\frac{a_1a_3a_5}{a_2a_4}, \,w=\sqrt{z-a_1}\sqrt{z-a_2}\sqrt{z-a_3}\sqrt{z-a_4}\sqrt{z-a_5}\sqrt{z-b_1},\\
g(z)&=\frac{\sqrt{z-a_1}\sqrt{z-a_3}\sqrt{z-a_5}}{\sqrt{z-a_2}\sqrt{z-a_4}\sqrt{z-b_1}},\, \text{and } \,\omega_3=\frac{1}{z}dz.
\end{split}
\]
Then there exists $\lambda_1,\lambda_2,\lambda_3,\lambda_4\in\C$ such that 
\[
\begin{split}
\omega_1&=\frac{1}{2}\left(\frac{1}{g(z)}-g(z)\right)\omega_3+\frac{\left(\lambda_1+\lambda_2z\right)dz}{w}\\
\omega_2&=\frac{i}{2}\left(\frac{1}{g(z)}+g(z)\right)\omega_3+\frac{\left(\lambda_3+\lambda_4z\right)dz}{w}\\
\omega_3&=\frac{1}{z}dz
\end{split}
\]
is the Weierstrass representation for a genus two doubly periodic harmonic surface with orthogonal top and bottom Scherk ends.  The handles lie between every other pair of consecutive ends.
\label{thm:scherk}
\end{theorem}

\begin{example}
The surface from theorem \ref{thm:scherk} will be a minimal surface when $\lambda_k=0$ for $k=1,2,\ldots,8$.  In this case, the period problem is solved when
\[
a_3=1, a_4=\frac{1}{a_2},\,a_5=\frac{1}{a_1}
\]
and
\[
(a_1,a_2)\approx(0.12539914,0.25068715).
\]
See figure \ref{figure:g2minscherk}.
\begin{figure}[h]
	\centerline{ 
		\includegraphics[height=2in]{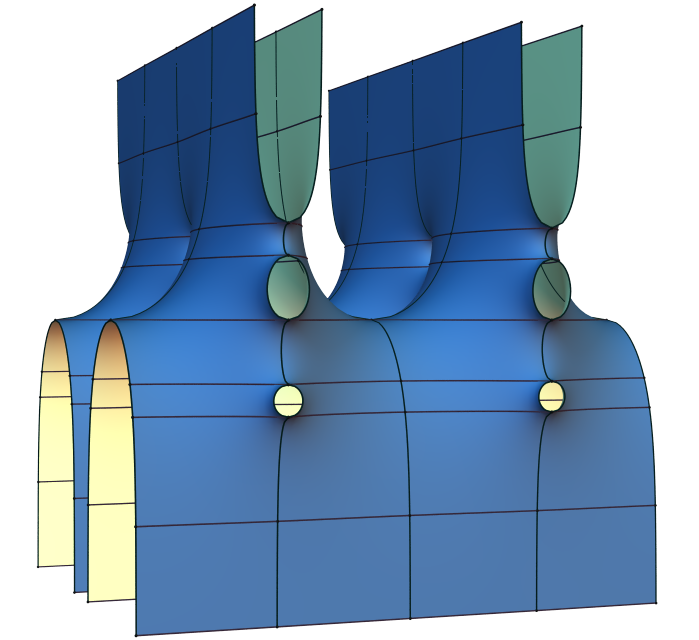}
			}
	\caption{Genus two minimal Scherk's surface with handles between every other pair of ends}
	\label{figure:g2minscherk}
\end{figure}
\end{example}

\begin{example}
If $(a_1,a_2,a_3,a_4,a_5)=(0.04,0.08,0.3,0.9,10)$ then the period problem is solved when 
\[
\lambda_1=0.13245255,\, \lambda_2=-1.37068243,\,\lambda_3=-0.36951468i,\, \lambda_4=-0.36951468i.
\]
See figure \ref{figure:scherkex1}.

If $(a_1,a_2,a_3,a_4,a_5)=(0.01,0.5,1,5,10)$ then the period problem is solved when 
\[
\lambda_1=1.65200244,\, \lambda_2=-1.62539483,\,\lambda_3=-1.18838503i,\, \lambda_4=-1.18838503i.
\]
See figure \ref{figure:scherkex2}.
\end{example}

\begin{figure}[h]
    \centering
    \begin{subfigure}[b]{0.48\textwidth}
    	    \centering
        \includegraphics[width=.75\textwidth]{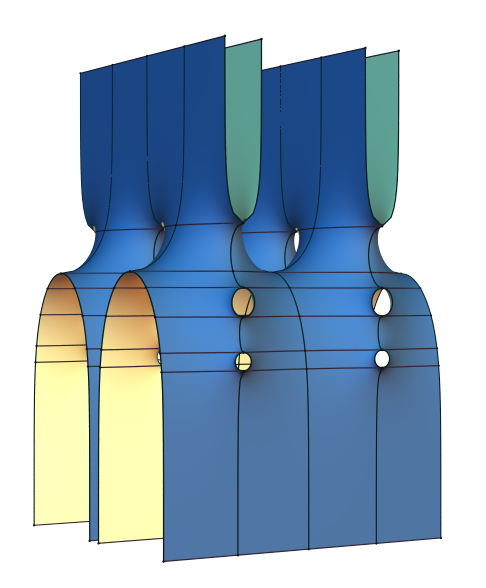}
        \caption{{$(a_1,a_2,a_3,a_4,a_5)= (.04,.08,.3,.9,10)$}}
        \label{figure:scherkex1}
    \end{subfigure}
    ~ 
        \begin{subfigure}[b]{0.48\textwidth}
    	    \centering
        \includegraphics[width=.75\textwidth]{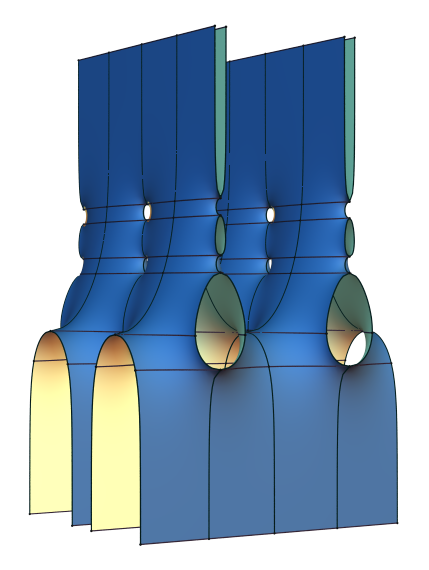}
        \caption{$(a_1,a_2,a_3,a_4,a_5)=(.01,.5,1,5,10)$}
        \label{figure:scherkex2}
    \end{subfigure}
    \caption{Genus two harmonic Scherk's surfaces with handles between every other pair of ends}\label{figure:g2scherk}
\end{figure}

Alternatively, we can change the direction of a handle.  The surfaces in figure \ref{figure:g2scherk} have handles that open in the same direction.  We could adjust the Weierstrass representation so the upper handle opens inward.  In that case, we switch the value of $g$ at $a_3$ and $a_4$ from the previous example, which forces the normal at $a_2$ and $a_3$ to be $(0,0,1)$.  The Weierstrass representation is

\[
\begin{split}
g(z)&=\frac{\sqrt{z-a_1}\sqrt{z-a_4}\sqrt{z-a_5}}{\sqrt{z-a_2}\sqrt{z-a_3}\sqrt{z-b_1}}\\
\omega_3&=\frac{1}{z}dz,
\end{split}
\]
with 
\[
b_1=-\frac{a_1a_4a_5}{a_2a_3}.
\]

In this case, the period problem is given by the equations
\[
\begin{split}
\re\int_{a_1}^{a_2}\omega_1&=\re\int_{a_2}^{a_3}\omega_1=\re\int_{a_3}^{a_4}\omega_1=\re\int_{a_4}^{a_5}\omega_1=0\\
\re\int_{a_1}^{a_2}\omega_2&=\re\int_{a_3}^{a_4}\omega_2=0\\
\re\int_{a_2}^{a_3}\omega_2&=-\re\int_{a_4}^{a_5}\omega_2=\pi,
\end{split}
\]
and it can be solved by adding multiples of the holomorphic forms
\[
\frac{dz}{w},\,\frac{zdz}{w}
\]
to $\omega_1$ and $\omega_2$, proving the following theorem.
\begin{theorem}
Let $(a_1,a_2,a_3,a_4,a_5)\in\R^5_+$ with $0<a_1<a_2<a_3<a_4<a_5$,
\[
\begin{split}
b_1&=-\frac{a_1a_4a_5}{a_2a_3}, \,w=\sqrt{z-a_1}\sqrt{z-a_2}\sqrt{z-a_3}\sqrt{z-a_4}\sqrt{z-a_5}\sqrt{z-b_1},\\
g(z)&=\frac{\sqrt{z-a_1}\sqrt{z-a_4}\sqrt{z-a_5}}{\sqrt{z-a_2}\sqrt{z-a_3}\sqrt{z-b_1}},\, \text{and } \,\omega_3=\frac{1}{z}dz.
\end{split}
\]
Then there exists $\lambda_1,\lambda_2,\lambda_3,\lambda_4\in\C$ such that 
\[
\begin{split}
\omega_1&=\frac{1}{2}\left(\frac{1}{g(z)}-g(z)\right)\omega_3+\frac{\left(\lambda_1+\lambda_2z\right)dz}{w}\\
\omega_2&=\frac{i}{2}\left(\frac{1}{g(z)}+g(z)\right)\omega_3+\frac{\left(\lambda_3+\lambda_4z\right)dz}{w}\\
\omega_3&=\frac{1}{z}dz
\end{split}
\]
is the Weierstrass representation for a genus two doubly periodic harmonic surface with orthogonal top and bottom Scherk ends.  There is one handle between every consecutive pair of ends.
\end{theorem}

\begin{example}
When $(a_1,a_2,a_3,a_4,a_5)=(0.0001,0.0009,0.006,0.02,10)$ the period problem is solved when 
\[
\lambda_1=-0.72642528,\, \lambda_2=-0.77146839,\,\lambda_3=3.38200891i,\, \lambda_4=3.43121095i.
\]
See figure \ref{figure:scherkex3}.

When $(a_1,a_2,a_3,a_4,a_5)=(0.00001,0.0001,0.006,0.2,0.8)$ the period problem is solved when 
\[
\lambda_1=-0.38387911,\, \lambda_2=-0.42144590,\,\lambda_3=-0.24540826i,\, \lambda_4=-0.19500638i.
\]
See figure \ref{figure:scherkex4}.
\end{example}
\begin{figure}[h]
    \centering
    \begin{subfigure}[b]{0.48\textwidth}
    	    \centering
        \includegraphics[width=.7\textwidth]{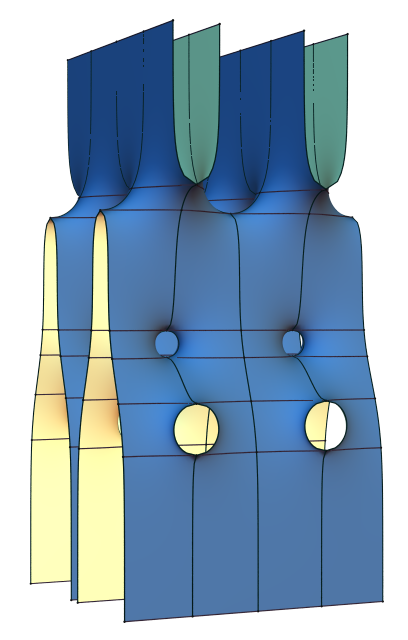}
        \caption{{$(a_1,a_2,a_3,a_4,a_5)=\\ (.0001,.0009,.006,.02,10)$}}
        \label{figure:scherkex3}
    \end{subfigure}
    ~ 
        \begin{subfigure}[b]{0.48\textwidth}
    	    \centering
        \includegraphics[width=.7\textwidth]{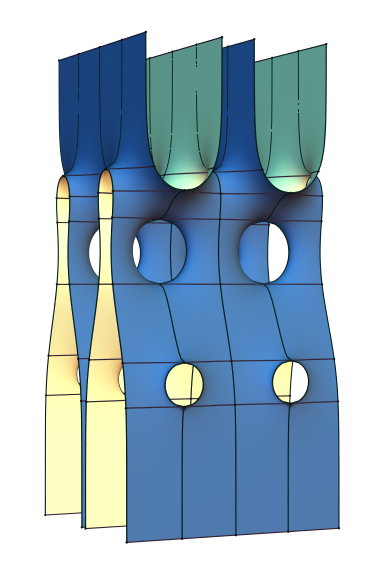}
        \caption{{$(a_1,a_2,a_3,a_4,a_5)=\\(.00001,.0001,.006,.2,.8)$}}
        \label{figure:scherkex4}
    \end{subfigure}
    \caption{Genus two harmonic Scherk's surfaces with handles between every other pair of ends}\label{figure:g2scherkb}
\end{figure}

\section{Harmonic tori}
\label{sec6}
If we restrict to harmonic tori then there is a great deal more flexibility for the placement of the handle, and we don't need to restrict to examples with two orthogonal symmetry planes as in the previous section.  We can, for example, construct a harmonic torus with one end of type $(2,2,4)$.  The example in figure \ref{figure:g0(2,2,4)} has a single end at $\infty$ of type $(2,2,4)$ and Weierstrass data $(\omega_1,\omega_2,\omega_3)=(1,-i,z^2)$.

Consider the rectangular torus $\Lambda$ spanned by $1$ and $\tau$.  Theta functions on $\Lambda$ of the form 
\[
\theta(z)=\sum_{n=-\infty}^{\infty}e^{\pi i(n_1/2)^2\tau+2\pi i(n+1/2)(z+1/2)}
\]
can be used to construct meromorphic one-forms on $\Lambda$.  As shown in \cite{we5} and \cite{wey1};
\begin{lemma}
Let $a_i,b_i\in\C,\; i=1,\ldots,n$.  Then
\[
h(z)=\prod_{i=1}^n\frac{\theta(z-a_i)^{\alpha_i}}{\theta(z-b_i)^{\beta_i}}
\]
has a zero of order $\alpha_i$ at $a_i$, a pole of order $\beta_k$ at $b_i$, and satisfies
\[
\begin{split}
h(z+1)&=(-1)^{\sum(\alpha_i-\beta_i)}f(z),\\
h(z+\tau)&=e^{2\pi i\sum(\alpha_ia_i-\beta_ib_i)}f(z).
\end{split}
\]
\end{lemma}

We can use this lemma to construct $\omega_1,\omega_2$, and $\omega_3$ on any torus corresponding to the Weierstrass data of an embedded harmonic tori with one end of type $(2,2,4)$.  The one-form $dz$ is holomorphic on $\Lambda$, and as discussed in section \ref{ssecperiod}, it can be used to solve the period problem.

\begin{figure}[h]
	\centerline{ 
		\includegraphics[height=2.5in]{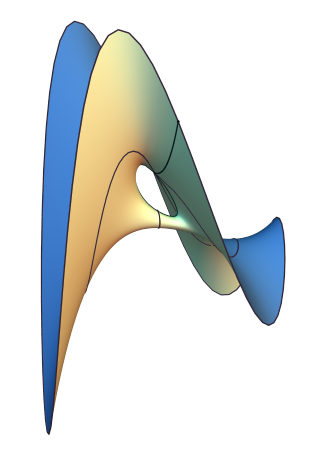}
		\hspace{.25in}
		\includegraphics[height=2.5in]{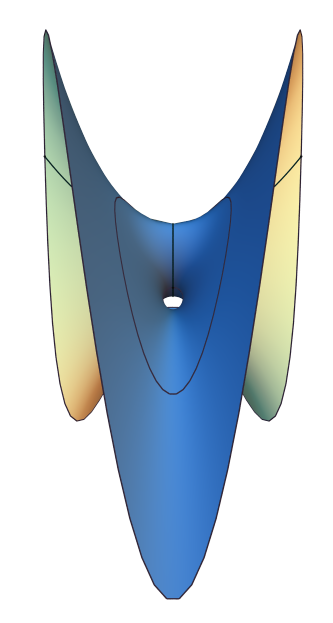}
			}
	\caption{Two views of a surface on $\Lambda$ with a $(2,2,4)$ end at $z=1/2$}
	\label{figure:g1(2,2,4)}
\end{figure}

\begin{theorem}
Let $\Lambda$ be the rectangular torus spanned by $1$ and $\tau=1.2i$.  There is a harmonic torus on $\Lambda$ with one end of type $(2,2,4)$ at $z=1/2$.  See figure \ref{figure:g1(2,2,4)}.
\end{theorem}

\begin{proof}
An end of type $(2,2,4)$ at $z=1/2$ is given by the one-forms
\[
\begin{split}
\frac{\omega_1}{dz}&=\frac{\theta(z-.5-.6i)\theta(z-.5+.6i)}{\theta(z-.5)^2}\\
\frac{\omega_2}{dz}&=\frac{i\theta(z-.5-.6i)\theta(z-.5+.6i)}{\theta(z-.5)^2}\\
\frac{\omega_3}{dz}&=\frac{i\theta(z-1-.6i)\theta(z+.6i)\theta(z)^2}{\theta(z-.5)^4}+\frac{i\theta(z-.5-.6i)\theta(z+.6i)\theta(z)}{\theta(z-.5)^3}-\frac{i\theta(z-.5-.6i)\theta(z-.5+.6i)}{\theta(z-.5)^2}.
\end{split}
\]
The higher order terms were added to $\omega_3$ to ensure the surface is embedded.  In order to solve the period problem, let 
\[
(\alpha_k,\beta_k)=\left(\re\int_{.3i}^{1+.3i}\omega_k,\im\int_0^{1.2i}\frac{\omega_k}{1.2i}\right),\, k=1,2,3.
\]
Then, 
\[
\begin{split}
\tilde{\omega}_1&=\omega_1+(\alpha_1+i\beta_1)dz\\
\tilde{\omega}_2&=\omega_2+(\alpha_2+i\beta_2)dz\\
\tilde{\omega}_3&=\omega_3+(\alpha_3+i\beta_3)dz
\end{split}
\] 
is the Weierstrass representation for a harmonic torus on $\Lambda$ with an end of type $(2,2,4)$ at $z=1/2$.
\end{proof}

Shifting the end from $z=1/2$ to $z=1/3$ slightly deforms the surface.  In this case, we don't need to add any higher order terms to $\omega_3$ to ensure the surface is embedded.

\begin{theorem}
Let $\Lambda$ be the rectangular torus spanned by $1$ and $\tau=1.2i$.  There is a harmonic torus on $\Lambda$ with one end of type $(2,2,4)$ at $z=1/3$.  See figure \ref{figure:g1(2,2,4)b}.
\end{theorem}

\begin{figure}[h]
	\centerline{ 
		\includegraphics[height=2.5in]{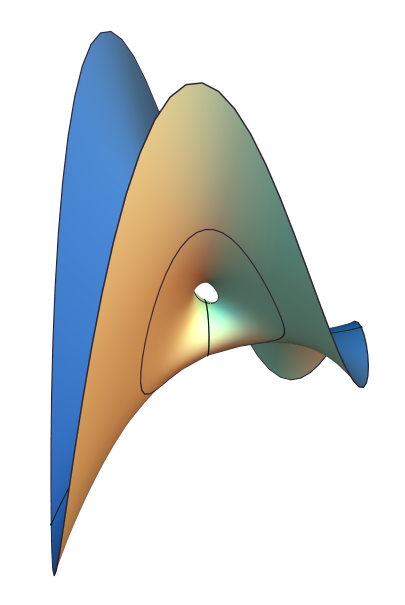}
		\hspace{.25in}
		\includegraphics[height=2.5in]{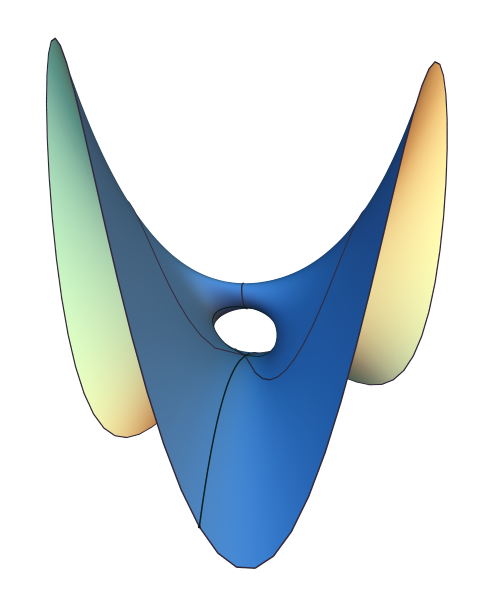}
			}
	\caption{Two views of a surface on $\Lambda$ with a $(2,2,4)$ end at $z=1/3$}
	\label{figure:g1(2,2,4)b}
\end{figure}

\begin{proof}
Let $\tau=1.2i$.  Place the end at $z=1/3$.  Let
\[
\begin{split}
\frac{\omega_1}{dz}&=\frac{\theta(z-1/3-.6i)\theta(z-1/3+.6i)}{\theta(z-1/3)^2}\\
\frac{\omega_2}{dz}&=\frac{i\theta(z-1/3-.6i)\theta(z-1/3+.6i)}{\theta(z-1/3)^2}\\
\frac{\omega_3}{dz}&=\frac{i\theta(z-1-.6i)\theta(z-1/3+.6i)\theta(z)^2}{\theta(z-.5)^4}.
\end{split}
\]
In order to solve the period problem, let 
\[
(\alpha_k,\beta_k)=\left(\re\int_{.3i}^{1+.3i}\omega_k,\im\int_0^{1.2i}\frac{\omega_k}{1.2i}\right),\, k=1,2,3.
\]
Then,
\[
\begin{split}
\tilde{\omega}_1&=\omega_1+(\alpha_1+i\beta_1)dz\\
\tilde{\omega}_2&=\omega_2+(\alpha_2+i\beta_2)dz\\
\tilde{\omega}_3&=\omega_3+(\alpha_3+i\beta_3)dz
\end{split}
\]
is the Weierstrass representation for a harmonic torus on $\Lambda$ with an end of type $(2,2,4)$ at $z=1/3$.
\end{proof}

    
\addcontentsline{toc}{section}{References}
\bibliographystyle{plain}
\bibliography{references}

\label{sec:liter}

%
\aua                                                        
    Peter Connor\\                                 
    Department of  Mathematical Sciences, Indiana University South Bend\\ 
    1700 Mishawaka Ave, South Bend, IN, 46634, USA\\        
    E-mail: pconnor@iusb.edu                           
%
%
\end{document}